\documentclass[]{amsart}
\usepackage{mathrsfs}
\usepackage{color}
\usepackage{amsmath}
\usepackage{amsfonts}
\usepackage{amssymb}
\usepackage{graphicx}
\usepackage{hyperref}

 \newtheorem{Theorem}{Theorem}[section]
 \newtheorem{Corollary}[Theorem]{Corollary}
 \newtheorem{Lemma}[Theorem]{Lemma}
 \newtheorem{Proposition}[Theorem]{Proposition}
 \newtheorem{Question}[Theorem]{Question}
 \newtheorem{Definition}[Theorem]{Definition}

 \newtheorem{Conjecture}[Theorem]{Conjecture}

 \newtheorem{Remark}[Theorem]{Remark}

 \numberwithin{equation}{section}

\begin{document}

\title[A remark on a weighted version of Suita conjecture for higher derivatives]
 {A remark on a weighted version of Suita conjecture for higher derivatives}

\author{Qi'an Guan}
\address{Qi'an Guan: School of Mathematical Sciences, Peking University, Beijing 100871, China.}
\email{guanqian@math.pku.edu.cn}

\author{Xun Sun}
\address{Xun Sun: School of Mathematical Sciences, Peking University, Beijing 100871, China.}
\email{sunxun@stu.pku.edu.cn}

\author{Zheng Yuan}
\address{Zheng Yuan: School of Mathematical Sciences, Peking University, Beijing 100871, China.}
\email{zyuan@pku.edu.cn}

\thanks{}



\date{}

\dedicatory{}

\commby{}


\begin{abstract}
In this article, we consider  the set of points for the holding of the equality in a weighted version of Suita conjecture for higher derivatives, and give relations between  the set and the integer valued points of a class of harmonic functions (maybe multi-valued). For planar domains bounded by finite analytic closed curves, we give relations between the set and Dirichlet problem.  
\end{abstract}

\maketitle

\section{introduction}
Let $\Omega$ be an open Riemann surface, which  admits a nontrivial Green function $G_{\Omega}$. Let $z_0\in\Omega$, and let $w$ be a local coordinate on a neighborhood 
$V_{z_0}$ of $z_0$ satisfying $w(z_0)$=0. Let $c_{\beta}(z)$ be the logarithmic capacity (see\cite{L-K}) on $\Omega$ which is defined by 
$$c_{\beta}(z_0):=\exp\lim_{z\rightarrow z_0}(G_{\Omega}(z,z_0)-\log|w(z)|).$$
Let $B_{\Omega}(z_0)$ be the Bergman kernel function on $\Omega$. An open question was posed by Sario and Oikawa in (\cite{L-K}):
find a relation between the magntitudes of the quantities $\sqrt{\pi B_{\Omega}(z)}$ and $c_{\beta}(z)$. 
In (\cite{Suita}), Suita posed the following conjecture.

\begin{Conjecture}
	 $\pi B_{\Omega}(z_0)\ge (c_{\beta}(z_0))^2$ holds, and the equality holds if and only if 
$\Omega$ is conformally equivalent to the unit disc less a (possible) closed set of inner capacity zero. 
\end{Conjecture}

The inequality part of Suita conjecture for bounded planar domain was proved by Blocki (\cite{Blocki}), and original form of inequality part was proved
by Guan-Zhou(\cite{G-ZhouL2_CR}). The equality part of Suita conjecture was proved by Guan-Zhou (\cite{guan-zhou13ap}), which completed the proof of Suita conjecture.

We recall some notations (see \cite{OF81}, see also \cite{guan-zhou13ap,G-Y}).
 Let $p:\Delta\rightarrow \Omega$ be the universal covering from unit disc $\Delta$ to $\Omega$.
 We call the holomorphic function $f$ on $\Delta$ a multiplicative function,
 if there is a character $\chi$, which is the representation of the fundamental group of $\Omega$, such that $g^{*}f=\chi(g)f$,
 where $|\chi|=1$ and $g$ is an element of the fundamental group of $\Omega$. Denote the set of such kinds of $f$ by $\mathcal{O}^{\chi}(\Omega)$.

It is known that for any harmonic function $u$ on $\Omega$,
there exist a $\chi_{u}$ and a multiplicative function $f_u\in\mathcal{O}^{\chi_u}(\Omega)$,
such that $|f_u|=p^{*}\left(e^{u}\right)$.
If $u_1-u_2=\log|f|$, then $\chi_{u_1}=\chi_{u_2}$,
where $u_1$ and $u_2$ are harmonic functions on $\Omega$ and $f$ is a holomorphic function on $\Omega$.
Recall that for the Green function $G_{\Omega}(z,z_0)$,
there exist a $\chi_{z_0}$ and a multiplicative function $f_{z_0}\in\mathcal{O}^{\chi_{z_0}}(\Omega)$ such that $|f_{z_0}(z)|=p^{*}\left(e^{G_{\Omega}(z,z_0)}\right)$ (see \cite{Suita}).

Let $u$ be a harmonic function on $\Omega$, and let $B_{\Omega,e^{-2u}}(z_0)$ be the Bergman kernel with the weight $e^{-2u}$. In \cite{yamada}, Yamada posed the following weighted version of Suita conjecture, which is so-called extended Suita conjecture.
\begin{Conjecture}
	$c_{\beta}^2(z_0)\le \pi e^{-2u(z_0)}B_{\Omega,e^{-2u}}(z_0)$, and equality holds if and only if $\chi_{z_0}=\chi_{-u}$.
\end{Conjecture}

 Note that $\Omega$ is conformally equivalent to the unit disc less a (possible) closed set of inner capacity zero if and only if $\chi_{z_0}\equiv1$ (see \cite{Suita}), thus the extended Suita conjecture is a general version of Suita conjecture. 
 In \cite{guan-zhou13ap}, Guan-Zhou proved the extended Suita conjecture (one can see more general versions  in \cite{G-M,G-Y}).

 Let us consider the weighted Bergman kernel for high derivatives (see \cite{Berg70,G-M-Y}). Let $k$ be a nonnegative integer, and let $I=(w)$ be an ideal of $\mathcal{O} _{z_0}$. Let $h=w^kdw$ be a holomorphic (1,0) form 
on $V_{z_0}$. Denote that 
 $$B_{\Omega,\rho}^{(k)}(z_0):=\frac{2}{\inf\left\{\int_{\Omega}|\widetilde{F}|^2\rho:\widetilde{F}\in H^0(\Omega,\mathcal{O}(K_{\Omega}))\,\&\,(\widetilde{F}-h,z_0)\in I^{k+1}\otimes\mathcal{O}(K_{\Omega})_{z_0}\right\}},$$
 where $\rho$ is a nonnegative Lebesgue measurable function on $\Omega$ and $|\widetilde{F}|^2:=\sqrt{-1}\widetilde F\wedge\overline{\widetilde F}$ for any  $(1,0)$ form $\widetilde F$. It is clear that 
 $$B_{\Omega,\rho}^{(k)}(z_0)=\sup\left\{\left|\frac{f^{(k)}(z_0)}{k!}\right|^2:\exists F\in H^0(\Omega,\mathcal{O}(K_{\Omega})),\,F|_{V_{z_0}}=fdw\,\&\,\frac{1}{2}\int_{\Omega}|F|^2\rho\le1\right\}.$$
 When $k=0$, $B_{\Omega,\rho}^{(k)}(z_0)$ degenerates to the weighted Bergman kernel $B_{\Omega,\rho}(z_0)$.  
 
 Then it is natural to ask
\begin{Question}
Can one prove a weighted version of Suita conjecture for higher derivatives related to the kernel function $B^{(k)}_{\Omega,\rho}(z_0)$?
\end{Question}

 Guan-Mi-Yuan \cite{G-M-Y} gave an affirmative answer to this question. Some related results can be referred to \cite{lz,x-z}.

Let $c(t)$ be a positive measurable function on $(0,+\infty)$ satisfying $c(t)e^{-t}$ is decreasing on $(0,+\infty)$ and $\int_{0}^{+\infty}c(t)e^{-t}dt<+\infty$. 
Denote that 
$$\rho_{z_0}:=e^{-(2(k+1-a)G_{\Omega}(\cdot,z_0)+2v)}c(-2aG_{\Omega}(\cdot,z_0)),$$ where $a>0$ is a constant and $v$ is a subharmonic function on $\Omega$.

\begin{Theorem}[\cite{G-M-Y}]
\label{thm3}
$$(c_{\beta}(z_0))^{2(k+1)}\le \left(\int_0^{+\infty}c(t)e^{-t}dt\right)\frac{\pi }{a}e^{-2v(z_0)}B_{\Omega,\rho_{z_0}}^{(k)}(z_0)$$
holds for any $z_0\in\Omega$ satisfying $v(z_0)>-\infty$.
Moreover, for $z_0\in\Omega$ with $v(z_0)>-\infty$, equality 
\begin{equation}
	\label{eq:221107}(c_{\beta}(z_0))^{2(k+1)}=\left(\int_0^{+\infty}c(t)e^{-t}dt\right)\frac{\pi }{a}e^{-2v(z_0)}B_{\Omega,\rho_{z_0}}^{(k)}(z_0)
	\end{equation}
holds if and only if the following statements hold: 

$(1)$ $v=\log|g|+u$, where $g$ is a holomorphic function on $\Omega$ such that $g(z_0)\neq 0$ and $u$ is a harmonic function on $\Omega$;

$(2)$ $\chi_{-u}=(\chi_{z_0})^{k+1}$, where $\chi_{-u}$ and $\chi_{z_0}$ are the characters associated to the functions $-u$ and $G_{\Omega}(\cdot,z_0)$ respectively.
\end{Theorem}

For Suita conjecture, if $\pi B_{\Omega}(z_0)= (c_{\beta}(z_0))^2$ holds for a point $z_0\in\Omega$, then it follows from the solution of Suita conjecture that $\pi B_{\Omega}(z_0)= (c_{\beta}(z_0))^2$ holds for any $z_0\in\Omega$, i.e., the set of points $z_0$ such that $\pi B_{\Omega}(z_0)= (c_{\beta}(z_0))^2$ holds is either $\emptyset$ or $\Omega$.

In this article, we discuss the set of points for the holding of the equality in the weighted version of Suita conjecture for higher derivatives.

\subsection{Main result}
\

Let $\Omega$ be an open Riemann surface, which admits a nontrivial Green function $G_{\Omega}$. Let $k$ be a nonnegative integer. Let $c(t)$ be a positive measurable function on $(0,+\infty)$ satisfying $c(t)e^{-t}$ is decreasing on $(0,+\infty)$ and $\int_{0}^{+\infty}c(t)e^{-t}dt<+\infty$. 
Denote that 
$$\rho_{z_0}:=e^{-(2(k+1-a)G_{\Omega}(\cdot,z_0)+2v)}c(-2aG_{\Omega}(\cdot,z_0)),$$ where $a>0$ is a constant and $v$ is a subharmonic function on $\Omega$.

It is clear that, if there is $\tilde g\in\mathcal{O}(\Omega)$ such that $v=\log|\tilde g|$  and $\Omega$ is conformally equivalent to the unit disc less a (possible) closed set of inner capacity zero, then $\chi_{z_0}=1$, which implies that equality \eqref{eq:221107} holds for any $z_0\in\Omega$ with $v(z_0)>-\infty$ by Theorem \ref{thm3}. The following theorem shows that the converse to the result is also true.

\begin{Theorem}
\label{thm4}Equality
\begin{equation}
	\label{eq:221129a}
(c_{\beta}(z_0))^{2(k+1)}=\left(\int_0^{+\infty}c(t)e^{-t}dt\right)\frac{\pi}{a}e^{-2v(z_0)}B_{\Omega,\rho_{z_0}}^{(k)}(z_0)
\end{equation}
holds for any $z_0\in\Omega$ satisfying  $v(z_0)>-\infty$ if and only if the following two statements hold:

$(1)$ $\Omega$ is conformally equivalent to the unit disc less a (possible) closed set of inner capacity zero;

$(2)$ $v=\log|g|$, where $g$ is a holomorphic function on $\Omega$ which is not identically equal to $0$.
\end{Theorem}

The following theorem gives a relation between  the set of points such that equality \eqref{eq:221129a} holds and the integer valued points of a class of harmonic functions (maybe multi-valued), which shows that the set of points where equality \eqref{eq:221129a} holds is a real analytic subset of $\Omega$ when one of the two statements in Theorem \ref{thm4} doesn't hold.
Let $\Pi_1\subset\pi_1(\Omega)$ such that $\pi_1(\Omega)$ is generated by $\Pi_1$.

\begin{Theorem}
	\label{remark6}Assume that $v=\log|\tilde g|+u$, where  $\tilde g\not\equiv0$ is a holomorphic function on $\Omega$ and $u$ is a harmonic function on $\Omega$. There exists a class of harmonic functions $\{u_{\alpha}\}_{\alpha\in\Pi_1}$ on $\Delta$, such that $\beta^*u_{\alpha}-u_{\alpha}$ is an integer constant function for any $\beta\in\pi_1(\Omega)$ and any $\alpha\in\Pi_1$. Equality
	\begin{equation}
		\label{eq:1118c}(c_{\beta}(z_0))^{2(k+1)}=\left(\int_0^{+\infty}c(t)e^{-t}dt\right)\frac{\pi}{a}e^{-2v(z_0)}B_{\Omega,\rho_{z_0}}^{(k)}(z_0)
	\end{equation}
holds for $z_0\in\Omega$ with $v(z_0)>-\infty$ if and only if  $u_{\alpha}(\tilde z_0)\in\mathbb{Z}$ for any $\alpha\in\Pi_1$ and $\tilde z_0\in p^{-1}(z_0)$.

Moreover,  $\{u_{\alpha}\}_{\alpha\in\Pi_1}$ are integer constant functions if and only if $\Omega$ is conformally equivalent to the unit disc less a (possible) 
closed set of inner capacity zero and $v=\log|g|$, where $g\not\equiv0$ is a holomorphic function on $\Omega$.
\end{Theorem}
\begin{Remark}
	In fact, $u_{\alpha}(z)$ is locally equal to $\frac{1}{2\pi}\left((k+1)\int_{\gamma_{\alpha}}\widetilde{d}G_{\Omega}(\cdot,p(z))+\int_{\gamma_{\alpha}}\widetilde{d}u\right)+m$ (see equality \eqref{eq:1118f}), where $\gamma_{\alpha}$ is a piecewise smooth closed curve on $\Omega$ satisfying $[\gamma_{\alpha}]=\alpha$, $\tilde{d}=\frac{\partial-\bar\partial}{i}$ and $m\in\mathbb{Z}$. Remark \ref{remark:example} in Appendix shows that $p_*u_{\alpha}$ is a single-valued function when $\Omega$ is a planar domain, but there exists an open Riemann surface such that $p_*u_{\alpha}$ may not be a single-valued function.
\end{Remark}

When $\rho_{z_0}\equiv1$ (i.e., $v\equiv0$, $a=k+1$ and $c\equiv1$), denote $B_{\Omega,\rho_{z_0}}^{(k)}(z_0)$ by $B_{\Omega}^{(k)}(z_0)$, which is the Bergman kernel for higher derivatives with trivial weight (see \cite{Berg70}). Theorem \ref{remark6} implies the following corollary.

\begin{Corollary}
There exists a class of harmonic functions $\{u_{\alpha}\}_{\alpha\in\Pi_1}$ on $\Delta$, such that $\beta^*u_{\alpha}-u_{\alpha}$ is an integer constant function for any $\beta\in\pi_1(\Omega)$ and $\alpha\in\Pi_1$. For any $z_0\in\Omega$, equality
$$(c_{\beta}(z_0))^{2(k+1)}=\frac{\pi}{k+1}B_{\Omega}^{(k)}(z_0)$$
holds if and only if  $u_{\alpha}(z_0)\in\mathbb{Z}$ for any $\alpha\in\Pi_1$.

Moreover,  $\{u_{\alpha}\}$ are integer constant functions if and only if $\Omega$ is conformally equivalent to the unit disc less a (possible) 
closed set of inner capacity zero.
\end{Corollary}

In the following, we assume that
 $\Omega$ is a bounded connected open subset of $\mathbb{C}$, which is bounded by $n$ analytic closed curves. 
 
 Denote that $\mathbb{C}\backslash \Omega=\bigcup_{i=1}^{n}A_j$, where $A_j$ are disjoint closed sets. 
Without loss of generality, suppose that $A_n$ is the unbounded  connected component of $\mathbb{C}\backslash\Omega$. 
Denote the boundary of $A_j$ by $\Gamma_j$, $1\le j\le n$.
For any $1\le j\le n-1$, let $\gamma_j$ be a closed smooth curve in $\Omega$, which winds around points in $A_j$ once and does not wind around points in $A_l$ for $l\neq j$ (for definition, see Lemma \ref{lemma8} and Remark \ref{remark5}).

For $1\le j\le n-1$, by solving Dirichlet problem (see Remark \ref{remark4}), there exists $u_j\in C(\overline{\Omega})$, which satisfies that $u_j=1$ on $\Gamma_j$, $u_j(z)=0$ on $\Gamma_l$ for any $l\neq j$ and $u_j$ is harmonic on $\Omega$.

The following theorem gives a relation between Dirichlet problem and the set of points such that equality \eqref{eq:1118c} holds.  

\begin{Theorem}
\label{thm5}
Equality
$$(c_{\beta}(z_0))^{2(k+1)}=\left(\int_0^{+\infty}c(t)e^{-t}dt\right)\frac{\pi }{a}e^{-2v(z_0)}B_{\Omega,\rho_{z_0}}^{(k)}(z_0)$$
holds for  $z_0\in\Omega$ with $v(z_0)>-\infty$ 
if and only if the following two statements hold:

$(1)$ $v=\log|g|+u$, where $g$ is a holomorphic function on $\Omega$ satisfying $g(z_0)\not=0$ and $u$ is a harmonic function on $\Omega$;

$(2)$ $(k+1)u_j(z_0)+\frac{1}{2\pi}\int_{\gamma_j}\frac{\partial u(z)}{\partial n}ds(z)\in \mathbb{Z}$ holds for any $1\le j\le n-1$, where $\frac{\partial}{\partial n}$ means differentiation in the direction of the outward pointing normal, and $s$ is the arc-length parameter.
\end{Theorem}

\begin{Remark}
	We call a domain $\Omega\subset\mathbb{C}$ is $n$-connected, if $\mathbb{C}\backslash\Omega=\cup_{j=1}^nA_j$, where $\{A_j\}$ are disjoint connected closed sets.
	By Lemma \ref{lemma7}, we know that Theorem \ref{thm5} also holds
	when $\Omega$ is a bounded $n$-connected open subset of $\mathbb{C}$ such that no connected component of $\mathbb{C}\backslash \Omega$ is an isolated point.
\end{Remark}

\begin{Remark}
\label{remark2}
We give an example of Theorem \ref{thm5}. Suppose that $\Omega=\{z\in\mathbb{C}:1<|z|<R\}$, and $v$ is a harmonic function on $\Omega$. Denote that 
$$c:=\int_{\{|z|=\frac{R+1}{2}\}}\frac{\partial v(z)}{\partial n}ds(z)$$
is a constant. Let $u_1=\frac{\log R-\log|z|}{\log R}$. It is clear that $u_1$ is harmonic on $\mathbb{C}\backslash\{0\}$, $u_1(z)=1$ when $|z|=1$ and $u_1(z)=0$ when $|z|=R$. Theorem \ref{thm5} shows that the set of points $z_0\in\Omega$ satisfying 
$$(c_{\beta}(z_0))^{2(k+1)}=\left(\int_0^{+\infty}c(t)e^{-t}dt\right)\frac{\pi }{a}e^{-2v(z_0)}B_{\Omega,\rho_{z_0}}^{(k)}(z_0)$$ 
is just $$\left\{z_0\in\Omega:|z_0|=R^{\frac{c+m}{k+1}} \text{ for some $m\in \mathbb{Z}$}\right\}.$$ 

Suppose that $\Omega$ is a doubly-connected open subset of $\mathbb{C}$ and no connected component of $\mathbb{C}\backslash\Omega$ is an isolated point. Let $v$ be a harmonic function on $\Omega$.
Note that $\Omega$ can be mapped conformally onto a circular ring (see \cite{nehari}), then the set of points $z_0\in\Omega$ satisfying 
$$(c_{\beta}(z_0))^{2(k+1)}=\left(\int_0^{+\infty}c(t)e^{-t}dt\right)\frac{\pi }{a}e^{-2v(z_0)}B_{\Omega,\rho_{z_0}}^{(k)}(z_0)$$ 
is constituted by $k$ real analytic closed curves.
\end{Remark}

When $\rho_{z_0}\equiv1$, Theorem \ref{thm5} implies the following corollary.
\begin{Corollary}
Equality
$$(c_{\beta}(z_0))^{2(k+1)}=\frac{\pi}{k+1}B_{\Omega}^{(k)}(z_0)$$
holds for  $z_0\in\Omega$
if and only if 
 $(k+1)u_j(z_0)\in \mathbb{Z}$  for any $1\le j\le n-1$.
\end{Corollary}

Under the assumption that $\Omega$ is a connected open subset of $\mathbb{C}$ bounded by $n$ analytic closed curves, the equality part of Suita conjecture is equivalent to that  ``$\pi B_{\Omega}(z_0)= (c_{\beta}(z_0))^2$ can not hold for any $z_0\in\Omega$ unless $n=1$". We generalize it to any nonnegative integer $k$ in the following corollary.
\begin{Corollary}
\label{coro1}
Assume that $v\equiv0$ on $\Omega$, then 
$$(c_{\beta}(z_0))^{2(k+1)}=\left(\int_0^{+\infty}c(t)e^{-t}dt\right)\frac{\pi }{a}B_{\Omega,\rho_{z_0}}^{(k)}(z_0)$$ 
can not hold for any $z_0\in\Omega$ unless $k\ge n-1$.
\end{Corollary}

\begin{Remark}
\label{remark7}
Suppose that $\Omega$ is bounded by $3$ analytic closed curves $\Gamma_j$ ($j=1,2,3$), and $v\equiv0$ on $\Omega$. Using Theorem \ref{thm5} and some discussions on harmonic functions, we get that there exists a sufficiently large $k$ and $z_0\in\Omega$ such that 
$$(c_{\beta}(z_0))^{2(k+1)}=\left(\int_0^{+\infty}c(t)e^{-t}dt\right)\frac{\pi }{a}B_{\Omega,\rho_{z_0}}^{(k)}(z_0).$$
 We prove the present remark in Section \ref{proof3}.
\end{Remark}

\section{preparation}
\label{section2}
In this section, we will do some preparations.\par
\subsection{Characters on open Riemann surfaces}\quad\par

Let $\Omega$ be an open Riemann surface, which admits a nontrivial Green function $G_{\Omega}$. Some basic knowledge about open Riemann surfaces  used in this section can be referred to \cite{OF81,nehari}. Let $p: \Delta\rightarrow\Omega$ be the universal covering from unit disc $\Delta$ to $\Omega$. Denote $Deck \left(\Delta\rightarrow\Omega\right)$
to be the set of all biholomorphic isomorphisms $\sigma:\Delta\rightarrow\Delta$ which preserve fibers, i.e., $p\circ \sigma=p$. We firstly recall the following lemma. 

\begin{Lemma}[see \cite{OF81}]
\label{lemma5}
Let $z_1$ be a point in $\Omega$. $Deck\left(\Delta\rightarrow\Omega\right)$ is isomorphic to the fundamental group $\pi_1\left(\Omega,z_1\right)$.
\end{Lemma}

\begin{proof}
Let $y_1\in\Delta$ be a point such that $p(y_1)=z_1$. For any $\sigma \in Deck\left(\Delta\rightarrow\Omega\right)$, we can construct a curve $\upsilon$ in $\Delta$ with initial point $y_1$ and end point $\sigma\left(y_1\right)$. Since $\Delta$ is simply connected, the homotopy class $[\upsilon]$ of $\upsilon$ is uniquely determined.
Then $p\circ\upsilon$ is a closed curve in $\Omega$ with endpoint $z_1$, and its homotopy class $[p\circ\upsilon]$ is uniquely determined, hence $[p\circ\upsilon]$ is a definite element in $\pi_1(\Omega,z_1)$. Now we can define a mapping
$$\Phi: Deck(\Delta\rightarrow\Omega)\rightarrow\pi_1(\Omega,z_1)$$
such that
$$\Phi\left(\sigma\right)=[p\circ\upsilon] \in \pi_1\left(\Omega,z_1\right)$$ 
for $\sigma\in Deck(\Delta\rightarrow\Omega)$. 

To prove that $\Phi$ is a group homomorphism, suppose that $\sigma, \tau\in Deck(\Delta\rightarrow\Omega)$, $\upsilon$ is a curve in $\Delta$ with initial point $y_1$ and end point $\sigma(y_1)$, and $\eta$ is a curve in $\Delta$ with initial point $y_1$ and end point $\tau(y_1)$. Then $\sigma\circ\eta$ is a curve in $\Delta$
with initial point $\sigma(y_1)$ and end point $\sigma\tau(y_1)$. Denote that $\upsilon\cdot(\sigma\circ\eta)$ is a curve on $\Delta$, which satisfies that $\upsilon\cdot(\sigma\circ\eta)(t)=\upsilon(2t)$ on $[0,\frac{1}{2}]$ and $\upsilon\cdot(\sigma\circ\eta)(t)=\sigma\circ\eta(2t-1)$ on $[\frac{1}{2},1]$. Note that $\upsilon\cdot(\sigma\circ\eta)$ has initial point $y_1$ and end point $\sigma\tau(y_1)$. Thus 
\begin{align}
\Phi(\sigma\tau)&=[p\circ(\upsilon\cdot(\sigma\circ\eta))]=[p\circ\upsilon][p\circ(\sigma\circ \eta)]\notag\\
&=[p\circ\upsilon][p\circ\eta]=\Phi(\sigma)\Phi(\tau)\notag.
\end{align}

To prove the injectivity of $\Phi$, suppose that $\Phi(\sigma)$ is trivial in $\pi_1(\Omega,z_1)$, then $p\circ\upsilon$ is null-homotopic. Since $\upsilon$ is the lifting of $p\circ\upsilon$, $\upsilon$ is also null-homotopic, hence $\sigma (y_1)=y_1$. $\sigma$ is therefore the identity mapping. 

To prove the surjectivity of $\Phi$, suppose that $\alpha$ is an element in $\pi_1(\Omega,z_1)$ and $u$ is a curve representing of $\alpha$. Let $\upsilon$ be the lifting of $u$ to $\Delta$ with initial point $y_1$, 
and denote the end point of $\upsilon$ by $y_2$. Then by the property of universal covering, there exists $\sigma\in Deck(\Delta\rightarrow\Omega)$ such that $\sigma (y_1)=y_2$. From the definition of $\Phi$, we have $\Phi(\sigma)=\alpha.$ 
\end{proof}

Let $h$ and $h_1$ be harmonic functions on an open subset $V$ of $\Omega$. We call $h_1$ is a harmonic conjugate of $h$ on $V$, if  $h+ih_1$ is a holomorphic function on $V$.
 Assume that $V$ is simply-connected, then for any harmonic function $h$ on $V$, there is a harmonic conjugate $\tilde h$ of $h$.  
Denote that 
$$\widetilde{d}=\frac{\partial-\bar{\partial}}{i}.$$

\begin{Lemma}
\label{lemma6}Let $z_1,z_2\in V$, and let $\gamma$ be a piecewise smooth curve in $V$ from $z_1$ to $z_2$. Then we have  
$$\tilde{h}(z_2)-\tilde{h}(z_1)=\int_{\gamma}\widetilde{d}h.$$
\end{Lemma}

\begin{proof} Without loss of generality, assume that $V$ is a simply-connected open subset of $\mathbb{C}$ and $\gamma$ is smooth (we can divide $\gamma$ into finite smooth curves which are contained in simply-connected coordinate open subsets of $\Omega$).

Note that
\begin{align} 
\widetilde{d}v&=\frac{1}{i}(\partial v -\bar{\partial} v)  \notag \\
&=\frac{1}{i}\left(\frac{\partial v}{\partial z}dz-\frac{\partial v}{\partial \bar{z}}d\bar{z}\right)\notag \\
&=\frac{1}{2i}\left(\left(\frac{\partial v}{\partial x}+\frac{1}{i}\frac{\partial v}{\partial y}\right)(dx+idy)-\left(\frac{\partial v}{\partial x}-\frac{1}{i}\frac{\partial v}{\partial y}\right)(dx-idy)\right)\notag\\
&=\frac{\partial v}{\partial x}dy-\frac{\partial v}{\partial y}dx\notag
\end{align}
holds for smooth function $v$, 
thus we know that 
$$\tilde h(z)=\int_{\gamma_z}\widetilde{d}h+C$$
 for any $z\in V$, where $\gamma_z$ is any smooth curve in $V$ from $z_1$ to $z$. As $h$ is harmonic and $V$ is simply-connected, we have 
 $$\tilde h(z_2)-\tilde h(z_1)=\int_{\gamma_{z_2}}\widetilde dh=\int_{\gamma}\widetilde{d}h.$$ 
\end{proof}

Let $z_0\in \Omega$, and let $u$ be a harmonic function on $\Omega$. There exist $f_{z_0},f_u\in\mathcal{O}(\Delta)$  such that $|f_{z_0}|=p^*e^{G_{\Omega}(\cdot,z_0)}$ and  $|f_u|=p^*e^u$. Let $z_1\in\Omega$ and $z_1\not=z_0$. By the isomorphism $\Phi: Deck(\Delta\rightarrow\Omega)\rightarrow\pi_1(\Omega,z_1)$ constructed in Lemma \ref{lemma5},  $\Phi^{-1}(g)$ is a biholomorphic isomorphism  on $\Delta$ which preserves fibers for any $g\in \pi_1(\Omega,z_1)$. Then we have 
$$\Phi^{-1}(g)^*f_{z_0}=\chi_{z_0}(g)f_{z_0}\,\,\text{ and }\,\,\Phi^{-1}(g)^*f_{u}=\chi_{u}(g)f_{u}.$$

Since $p:\Delta\rightarrow\Omega$ is the universal covering, any element in $\pi_1(\Omega,z_1)$ can be represented by $p\circ\widetilde{\gamma}$ where $\widetilde{\gamma}$ is a curve in $\Delta$, and it only depends on the starting point and end point of $\widetilde{\gamma}$ by Lemma \ref{lemma5}. Thus, any $g\in\pi_1(\Omega,z_1)$ can be represented by a piecewise smooth curve $\gamma$ in $\Omega$ satisfying $z_0\not\in \gamma$.
 
In the following proposition, we represent $\chi_{z_0}(g)$ and $\chi_{u}(g)$ by the integrals of $\widetilde{d}G_{\Omega}(\cdot,z_0)$ and $\widetilde{d}u$ along $\gamma$ respectively.

\begin{Proposition}
\label{thm1}
 Let $g$ be an element in $\pi_1(\Omega,z_1)$ which can be represented by $\gamma$, where $\gamma$ is a piecewise smooth curve in $\Omega$ and $z_0\not\in\gamma$. Then  we have
$$\chi_{z_0}(g)=e^{i\int_{\gamma}\widetilde{d}G_{\Omega}(\cdot,z_0)}$$
and
$$\chi_{u}(g)=e^{i\int_{\gamma}\widetilde{d}u}.$$
\end{Proposition}

\begin{proof}
\label{propo1}
Lift $\gamma$ to a  curve $\tilde\gamma$ in $\Delta$ with initial point $y_1$. 
Denote the end point of $\tilde\gamma$ by $y$. 
By the isomorphism $\Phi: Deck(\Delta\rightarrow\Omega)\rightarrow\pi_1(\Omega,z_1)$ constructed in Lemma \ref{lemma5},  $\Phi^{-1}(g)$ is a biholomorphic isomorphism $\sigma:\Delta\rightarrow\Delta$ which preserves fibers and $\sigma(y_1)=y$.
For any holomorphic function $f$ in $\Delta$, $\sigma^*f$ is a holomorphic function in $\Delta$ and $\sigma^{*}f(y_1)=f(y)$. As $\sigma^*f_{z_0}=\chi_{z_0}(g)f_{z_0}$, we have 
\begin{equation}
	\label{eq:1108a}f_{z_0}(y)=\sigma^*f_{z_0}(y_1)=\chi_{z_{0}}(g)f_{z_0}(y_1).
\end{equation}

Let $\{V_j\}_{j=1}^N$ be a simply-connected open covering of $\tilde\gamma$ satisfying that $p$ maps $V_j $ biholomorphically onto $p(V_j)$. We can divide $\tilde\gamma$ into $\gamma_j$ such that $\gamma_j\subset V_j$. Denote the initial point of $\gamma_j$ by $a_j$, the end point of $\gamma_j$ by $b_j$.

By Lemma \ref{lemma6}, for any $j$, we have 
$$f_{z_0}=p^*e^{(G_{\Omega}(\cdot,z_0)+i\widetilde{G_{\Omega}}(\cdot,z_0))}$$
on $V_j$
and 
$$\widetilde{G_{\Omega}}(p(b_j),z_0)-\widetilde{G_{\Omega}}(p(a_j),z_0)=\int_{p_*\gamma_j}\widetilde{d}G_{\Omega}(\cdot,z_0),$$ 
where $\widetilde{G_{\Omega}}(z,z_0)$ is the harmonic conjugate of $G_{\Omega}(z,z_0)$ on $p(V_j)$.  Therefore
\begin{align}
\frac{f_{z_0}(b_j)}{f_{z_0}(a_j)}&=e^{G_{\Omega}(p(b_j),z_0)+i\widetilde{G_{\Omega}}(p(b_j),z_0)-G_{\Omega}(p(a_j),z_0)-i\widetilde{G_{\Omega}}(p(a_j),z_0)}\notag\\
&=e^{G_{\Omega}(p(b_j),z_0)-G_{\Omega}(p(a_j),z_0)+i\int_{p_*\gamma_j} \widetilde{d}G_{\Omega}(z,z_0)}.\notag
\end{align}
As $p(a_1)=p(b_N)$, we have
\begin{align}
\frac{f_{z_0}(b_N)}{f_{z_0}(a_1)}&=\prod_{j=1}^N\frac{f_{z_0}(b_j)}{f_{z_0}(a_j)}\notag\\
&=e^{\sum_{j=1}^N(G_{\Omega}(p(b_j),z_0)-G_{\Omega}(p(a_j),z_0)+i\int_{p_*\gamma_j} \widetilde{d}G_{\Omega}(\cdot,z_0))}\notag\\
&=e^{G_{\Omega}(p(b_N),z_0)-G_{\Omega}(p(a_1),z_0)+i\int_{\gamma} \widetilde{d}G_{\Omega}(\cdot,z_0)}\notag\\
&=e^{i\int_{\gamma}\widetilde{d}G_{\Omega}(\cdot,z_0)}.\notag
\end{align}
Note  that $a_1=y_1$, $b_N=y$ and $p(a_1)=p(b_N)=z_1$, thus $$f_{z_0}(y)=f_{z_0}(y_1)e^{i\int_{\gamma}\widetilde{d}G_{\Omega}(\cdot,z_0) }.$$ 
Combining this with equality \eqref{eq:1108a}, we therefore have 
$$\chi_{z_0}(g)=e^{i\int_{\gamma}\widetilde{d}G_{\Omega}(\cdot,z_0)}.$$ 
Similarly, we have 
$$\chi_{u}(g)=e^{i\int_{\gamma}\widetilde{d}u}.$$
\end{proof}

\subsection{Some lemmas on planar domains}\quad\par
In this section, we consider the case that the open Riemann surface $\Omega$ is a bounded connected open subset of $\mathbb{C}$ such that $\mathbb{C}\backslash\Omega=\bigcup_{j=1}^{n} A_j$, where $\{A_j\}$ are disjoint connected closed sets and each $A_j$ is not an isolated point.   We recall the following well-known lemma.
\begin{Lemma}
\label{lemma7}
Assume that $\Omega$ is a bounded $n$-connected open subset of $\mathbb{C}$ such that no connected component of $\mathbb{C}\backslash \Omega$ is an isolated point, then $\Omega$ is conformally equivalent to an open subset of $\mathbb{C}$ bounded by $n$ closed analytic curves.
\end{Lemma}
\begin{proof}
	For convenience of readers, we give a proof of this lemma by using Riemann mapping theorem.
	
Assume that $\mathbb{C}\backslash\Omega=\bigcup_{j=1}^n A_j$, where $A_j$ are disjoint connected closed sets. Denote the boundary of $A_j$ by $C_j$. Without loss of generality, assume that $A_n$ is unbounded.

Let $A=\Omega\bigcup(\bigcup_{j=2}^{n}A_j)\cup\{\infty\}=\mathbb{CP}^1\backslash A_1$, which is a simply-connected proper open subset of $\mathbb{CP}^1$ whose boundary $C_1$ is not an isolated point.
Then by Riemann mapping theorem, $A$ can be mapped conformally onto the unit disc $\Delta$. This mapping can transform $\Omega$ onto $\Omega_1=\Delta\backslash (\bigcup_{j=2}^{n}A_j')$, 
where $A_j'$ is the conformal image of $A_j$ and $C_j'=\partial A _j'$ is the conformal image of $C_j$ for $2\le j\le n$. In further let $A_1'=\mathbb{C}\backslash\Delta$, $C_1'=\partial A_1'=\{z\in\mathbb{C}:|z|=1\}$.

Repeating the above procedure for another $n-1$ times, we can transform $\Omega$ onto $\Omega_n=\Delta\backslash(\bigcup_{j=1}^{n-1}A_j^{(n)})$, where $C_j^{(n)}=\partial A_j^{(n)}$ are closed analytic curves for $1\le j\le n-1$.  It is obvious that $\Omega_n$ is a bounded open subset of $\mathbb{C}$ which is bounded by $n$ closed analytic curves.   
\end{proof}

 We recall the following basic knowledge in algebraic topology.
 
\begin{Lemma}[see \cite{fulton}]\label{lemma8} Let $U$ be a bounded connected open subset of $\mathbb{C}$ such that $\mathbb{C}\backslash U=\bigcup_{j=1}^m A_j$, where $A_j$ are disjoint connected closed sets and $A_m$ is unbounded. For each $1\le j\le m-1$, there exists a closed smooth curve $\gamma_j$ in $\Omega$ which winds around points in $A_j$
once and does not wind around points in $A_l$ for $l\neq j$. If $\gamma_j'$ is an another closed smooth curve which satisfies these conditions, then $\gamma_j'$ is homologous to $\gamma_j$. The homology classes of these $m-1$ curves form a basis for the first homology group $H_1(U)$. That is, these curves are not homologous to each other and any curve in $U$ is homologous to some linear combinations of these curves. 
\end{Lemma}

\begin{Remark}
\label{remark5}
The times that a closed smooth curve $\gamma$ winds around a point $P=(x_0,y_0)\in\mathbb{R}^2$ is defined as follows. In $\mathbb{R}^2\backslash (x_0,y_0)$ we can define a differential $\omega_P=\frac{-(y-y_0)dx+(x-x_0)dy}{(x-x_0)^2+(y-y_0)^2}$, and define $W(\gamma,P)=\frac{1}{2\pi}\int_{\gamma}\omega_P$ to be the times the curve $\gamma$ winds around the point $P$.
For example, if $\gamma$ is the curve $\gamma:[0,1]\rightarrow\mathbb{R}^2\backslash(x_0,y_0), \gamma(t)=(x_0+r_0\cos 2\pi t,y_0+2\sin\pi t)$, then $W(\gamma,P)=1$. 
For arbitrary closed curve, $W(\gamma,P)$ is always an integer. For details, see \cite{fulton}.
\end{Remark}

\begin{Lemma}
	\label{l:1}$\frac{\partial v}{\partial n}ds=\widetilde{d}v$ holds for smooth function $v$, where $\frac{\partial}{\partial n}$ means differentiation in the direction of the outward pointing normal, $s$ is the arc-length parameter and $\widetilde d=\frac{\partial-\bar\partial}{\sqrt{-1}}$.
\end{Lemma}

\begin{proof}
Note that 
$$\widetilde{d}v=\frac{\partial v}{\partial x}dy-\frac{\partial v}{\partial y}dx$$
and  
$$\frac{\partial v}{\partial n}=v_x\cos(x,n)+v_y\cos(y,n),$$
where $(x,n)$ and $(y,n)$ denote the angle between the outward pointing normal and the positive $x$- and $y$-axis, respectively. Since the cosines of the angles between the tangent direction and the positive $x$- and $y$-axis are $\frac{dx}{ds}$ and $\frac{dy}{ds}$, where $s$ is the arc-length parameter,
and the outer normal is obtained by turning the tangent clockwise by $\frac{\pi}{2}$, it follows that
$$\cos(x,n)=\frac{dy}{ds},\quad \cos(y,n)=-\frac{dx}{ds}.$$
Hence
$$\frac{\partial v}{\partial n}ds=v_xdy-v_ydx=\widetilde{d}v. $$
\end{proof}

Because of Lemma \ref{lemma7}, in the following part, we assume that
 $\Omega$ is a bounded connected open subset of $\mathbb{C}$, which is bounded by $n$ analytic closed curves.

Let $u$ be a harmonic function on $\Omega$, and let  $z_0\in\Omega$. It follows from Lemma \ref{lemma8} that, for any $1\le l\le n-1$, there exists a closed smooth curve $\gamma_l$ in $\Omega$ which winds around points in $A_l$
once, does not wind around points in $A_j$ for $j\neq l$ and satisfies $z_0\not\in\gamma_l$.
\begin{Lemma}
\label{lemma2}
$(\chi_{z_0})^{k+1}=\chi_{-u}$ holds
if and only if 
$$\left(e^{i\int_{\gamma_j}\frac{\partial G_{\Omega}(z,z_0)}{\partial n}ds(z)}\right)^{k+1}=e^{-i\int_{\gamma_j}\frac{\partial u(z)}{\partial n}ds(z)}$$ holds for $1\le j\le n-1$. 
\end{Lemma}

\begin{proof}

It follows from Proposition \ref{thm1} that $(\chi_{z_0})^{k+1}=\chi_{-u}$ holds
if and only if 
$$\left(e^{i\int_{\gamma}\frac{\partial G_{\Omega}(z,z_0)}{\partial n}ds(z)}\right)^{k+1}=e^{-i\int_{\gamma}\frac{\partial u(z)}{\partial n}ds(z)}$$ holds for smooth closed curve $\gamma$ in $\Omega\backslash\{z_0\}$. 

Let $\gamma_0$ be a smooth closed curve in $\Omega\backslash\{z_0\}$ which winds around $z_0$ once and does not wind around points in $\mathbb{C}\backslash\Omega$. By Lemma \ref{lemma8}, the first homology group $H_1(\Omega\backslash\{z_0\})\cong \mathbb{Z}^{\oplus n}$, and the homology classes of $\gamma_j$ for $0\le j\le n-1$ form its basis. 
Hence for any smooth curve $\gamma\subset\Omega\backslash \{z_0\}$, there exists a sequence of integers $t_j$ for $0\le j\le n-1$ such that $\sum_{j=0}^{n-1}t_j \gamma_j$ is homologous to $\gamma$.

 Since $G_{\Omega}(z,z_0)$ and $u$ are harmonic on $\Omega\backslash\{z_0\}$,  $\frac{\partial G_{\Omega}(z,z_0)}{\partial n}ds$ and $\frac{\partial u(z)}{\partial n}ds$ are closed differential 1-forms on $\Omega\backslash\{z_0\}$ by Lemma \ref{l:1}. Then we have   
$$\int_{\gamma}\frac{\partial G_{\Omega}(z,z_0)}{\partial n}ds=\sum_{j=0}^{n-1}t_j\int_{\gamma_j}\frac{\partial G_{\Omega}(z,z_0)}{\partial n}ds$$
and
$$\int_{\gamma}\frac{\partial u(z)}{\partial n}ds=\sum_{j=0}^{n-1}t_j\int_{\gamma_j}\frac{\partial u(z)}{\partial n}ds.$$
Recall that $G_{\Omega}(z,z_0)=\log|z-z_0|+u'(z)$ on $\Omega$, where $u'$ is a harmonic function on $\Omega$, hence 
$$\int_{\gamma_0}\frac{\partial G_{\Omega}(z,z_0)}{\partial n}ds=\int_{\gamma_0}\frac{\partial \log|z-z_0|}{\partial n}ds+\int_{\gamma_0}\frac{\partial u'(z)}{\partial n}ds=2\pi$$
and
$$\int_{\gamma_0} \frac{\partial u(z)} {\partial n}ds=0.$$ 
Therefore we get that 
$$\left(e^{i\int_{\gamma}\frac{\partial G_{\Omega}(z,z_0)}{\partial n}ds(z)}\right)^{k+1}=e^{-i\int_{\gamma}\frac{\partial u(z)}{\partial n}ds(z)}$$ holds for smooth closed curve $\gamma$ in $\Omega\backslash\{z_0\}$ if and only if 
$$\left(e^{i\int_{\gamma_j}\frac{\partial G_{\Omega}(z,z_0)}{\partial n}ds}\right)^{k+1}=e^{-i\int_{\gamma_j}\frac{\partial u(z)}{\partial n}ds}$$
holds for $1\le j\le n-1$. 
Thus, Lemma \ref{lemma2} has been proved.
\end{proof}

\subsection{Harmonic functions and Dirichlet problem}\quad\par

Firstly, we recall two properties of harmonic functions on an open subset of $\mathbb{R}^2$.

\begin{Lemma}[Harnack's inequality, see \cite{evans}]\label{lemma3}

Let $V,U$ be two connected open subsets of $\mathbb{R}^2$ such that $V\Subset  U$, then there exists a positive constant $C$, which depends only on $V$ and $U$, such that 
$$\mathop{\sup}\limits_V u \le C\mathop{\inf}\limits_V u$$ 
for all nonnegative harmonic functions $u$ on $U$.
\end{Lemma}

\begin{Lemma}[see \cite{evans}]\label{lemma4}
 Assume $u$ is harmonic on $U\subset \mathbb{R}^n$. Then for any positive integer $k$, there exists a constant $C_k$ which is independent of $u$ and $U$, such that
$$|D^{\alpha}u(x_0)|\le \frac{C_k}{r^{2+k}}\vert u\vert _{L^1(B(x_0,r))} $$
for each ball $B(x_0,r)\subset U$ and each multi-index $\alpha$ of order $|\alpha|=k$.
\end{Lemma}

The following lemma will be used in the proof of Remark \ref{remark7}.

\begin{Lemma}
	\label{l:u1u2}Let $u_j$ be the harmonic function on a unit disc $\Delta\subset\mathbb{C}$, where $j=1,2$. Assume that $u_j$ is not a constant function for $j=1,2$, and $u_2\equiv c_1u_1+c_2$  does not hold for any constant $c_1,c_2$. Then there exists $z_0\in\Delta$ such that $u_1(z_0)\in\mathbb{Q}$ and $u_2(z_0)\in\mathbb{Q}$.
\end{Lemma}
\begin{proof}
As $\Delta$ is simply connected, there exists a harmonic function $v_1$ on $\Delta$ such that $f:=u_1+iv_1$ is holomorphic on $\Delta$. As $u_1$ is not a constant function, shrinking $\Delta$ if necessary, we can assume that $f$ is a coordinate function on $\Delta$  without loss of generality. Denote that $h=u_2\circ f^{-1}$ on $f(\Delta)$, then $h$ is a harmonic function on $f(\Delta)$. 
	 
Now, we prove Lemma \ref{l:u1u2} by contradiction: if not, $u_2(z)\not\in\mathbb{Q}$ for any $z\in\Delta$ with $u_1(z)\in\mathbb{Q}$. As $u_2$ is continuous, we know $u_2\equiv const$ on $\{z\in\Delta:u_1(z)=c\}$, i.e., 
$$\frac{\partial u_2}{\partial v_1}=0$$
on $\{z\in\Delta:u_1(z)=c\}$, where $c\in\mathbb{Q}$. Note that $\cup_{c\in\mathbb{Q}}\{z\in\Delta:u_1(z)=c\}$ is dense in $\Delta$, then we have 
$\frac{\partial u_2}{\partial v_1}=0$
 on $\Delta$, which implies that $h(z)$ is only dependent on $\text{Re}z$. As $h$ is a harmonic function on $f(\Delta)$, there exist constants $c_1,c_2$ such that 
 $$u_2=h(u_1+iv_1)=c_1u_1+c_2,$$ which contradicts to the assumption in Lemma \ref{l:u1u2}.

Thus, Lemma \ref{l:u1u2} has been proved.
\end{proof}

We recall some results on Dirichlet problem, which will be used in the proofs of our main results.

 Dirichlet problem: 
\emph{Let $\Omega$ be an open subset of $\mathbb{R}^n$, and let $U(z)$ be a continuous function on $\partial \Omega$. Does there exist a continuous function $u$ on $\overline{\Omega}$, which is harmonic on $\Omega$ and satisfies $u|_{\partial\Omega}= U$?}

\

We recall the following definition. 
\begin{Definition}
Let $\Omega\subset\mathbb{R}^n$ be a bounded domain. A boundary point $z\in\partial\Omega$ is called regular, if there is a continuous function $\phi$ on $\overline\Omega$ such that $\phi$ is subharmonic on $\Omega$, $\phi(z)=0$ and $\phi<0$ on $\partial \Omega\backslash\{z\}$.
\end{Definition}

The following lemma gives a sufficient condition for solvability of Dirichlet problem.

\begin{Lemma}[Perron 1923, see \cite{OF81}]\label{lemma10}
Let $\Omega\subset\mathbb{R}^n$ be a bounded domain. Let $U$ be a function on $\partial\Omega$. Then there exists a harmonic function $u$ on $\Omega$, such that if $z\in\partial\Omega$ is a regular point and $U$ is continuous at $z$, $u(x)\rightarrow U(z)$ as $\Omega\ni x\rightarrow z$.
\end{Lemma}

In the following, assume that $\Omega\subset\mathbb{R}^2 $. The following lemma gives a  useful criterion for $p\in\partial\Omega$ to be regular.

\begin{Lemma}[Osgood 1900, see \cite{Ahlfors74}]
\label{lemma11}
 Let $\Omega\subset\mathbb{R}^2$ be a bounded domain. Suppose that $p\in\partial\Omega$ is contained in a component of $\mathbb{R}^2\backslash\Omega$ which has more than one point, that is, $p$ is not an isolated point in $\mathbb{R}^2\backslash\Omega$, then $p$ is regular.
\end{Lemma}

\begin{Remark}
	\label{remark4}
	Let $\Omega$ be a bounded connected open subset of $\mathbb{C}$ such that $\mathbb{C}\backslash \Omega=\bigcup_{i=1}^{n}A_j$, where $A_j$ are disjoint connected closed sets and each $A_j$ is not an isolated point.
	Then by Lemma \ref{lemma10} and Lemma \ref{lemma11}, Dirichlet problem on such Riemann surface is always solvable.
\end{Remark}

Let us recall a basic formula, using that we can obtain the solution of  Dirichlet problem.

\begin{Lemma}[Green third formula, see \cite{nehari}]\label{lemma1}
Assume that $\Omega$ is bounded by closed smooth curve $\Gamma$. Let $u\in C(\overline{\Omega})$, which is  harmonic on $\Omega$.  Then we have
$$u(z_0)=\frac{1}{2\pi}\int_{\Gamma}u\frac{\partial G_{\Omega}(z,z_0)}{\partial n}ds.$$ 
\end{Lemma}

\section{Proofs of Theorem \ref{thm4} and Theorem \ref{remark6}}
In this section, we prove Theorem \ref{thm4} and Theorem \ref{remark6}.

\subsection{Proof of Theorem \ref{thm4}.}\quad\par
Suppose that $\Omega$ is conformally equivalent to unit disc less a (possible) closed set of inner capacity zero. For every $z_0\in\Omega$, $\chi_{z_0}=1$, hence $\chi_{z_0}^{k+1}=1$ for every integer $k\ge0$. If 
$v=\log|g|$, where $g\not\equiv0$ is a holomorphic function on $\Omega$, then by Theorem \ref{thm3}, 
$$(c_{\beta}(z_0))^{2(k+1)}=\left(\int_0^{+\infty}c(t)e^{-t}dt\right)\frac{\pi }{a}e^{-2v(z_0)}B_{\Omega,\rho_{z_0}}^{(k)}(z_0)$$ 
holds for every $z_0\in\Omega\backslash g^{-1}(0)$. 

On the other hand,  suppose that
$$(c_{\beta}(z_0))^{2(k+1)}=\left(\int_0^{+\infty}c(t)e^{-t}dt\right)\frac{\pi }{a}e^{-2v(z_0)}B_{\Omega,\rho_{z_0}}^{(k)}(z_0)$$
holds for every point $z_0\in\Omega$ such that $v(z_0)>-\infty$. Let $z_0'\in\Omega$ be a point such that $v(z_0')>-\infty$. By Theorem \ref{thm3},
 we have
$$v=\log|\tilde g|+u,$$ 
where $\tilde g\not\equiv0$ is a holomorphic function on $\Omega$ such that $\tilde g(z_0')\neq 0$, $u$ is a harmonic function on $\Omega$ and $(\chi_{z_0'})^{k+1}=\chi_{-u}.$ 
Notice that for every $z_0\in\Omega\backslash \tilde g^{-1}(0)$, we have $v(z_0)>-\infty$, hence 
\begin{equation}
	\label{eq:1109a}(\chi_{z_0})^{k+1}=\chi_{-u}
	\end{equation}
 for all such $z_0$.                                                                                                                                                                                                                       

Let $\gamma$ be any closed piecewise smooth curve in $\Omega$.
According to Proposition \ref{thm1} and equality \eqref{eq:1109a}, for every $z_0\in\Omega\backslash \gamma$ such that $\tilde g(z_0)\not=0$,
$$\left(e^{i\int_{\gamma}\widetilde{d}G_{\Omega}(\cdot,z_0)}\right)^{k+1}=e^{-i\int_{\gamma}\widetilde{d}u},$$
hence there exists an integer $j_{z_0}$ such that 
\begin{equation}
	\label{eq:1109b}\int_{\gamma}\widetilde{d}G_{\Omega}(\cdot,z_0)=\frac{1}{k+1}\left(2\pi j_{z_0}-\int_{\gamma}\widetilde{d}u\right).
\end{equation}
By Lemma \ref{lemma14} in Appendix, there exists a harmonic function $H$  on $\Delta$ such that
\begin{equation}\nonumber
\frac{1}{2\pi}\left(H(z)-\int_{\gamma}\widetilde{d}G_{\Omega}(\cdot,p(z))\right)\in\mathbb{Z}
\end{equation} for any $z\in \Delta\backslash p^{-1}(\gamma)$, where $p:\Delta\rightarrow\Omega$ is the universal covering.
It follows from equality \eqref{eq:1109b} that 
$$\frac{1}{2\pi}\left((k+1)H(z)+\int_{\gamma}\widetilde{d}u\right)\in\mathbb{Z}$$
for  any $z\in \Delta\backslash p^{-1}(\gamma)$, which shows that $H(z)$ is a constant function (the constant denote also by $H$) on $\Delta$. Then we get that 
 \begin{equation}
 	\label{eq:1118a}G_{\gamma}(z):=\int_{\gamma}\widetilde{d}G_{\Omega}(\cdot,z)=H+2m\pi
 \end{equation}
 is a constant function on any connected component of $\Omega\backslash \gamma$, where $m$ is a constant dependent on the selection  of the components. 

Fixing a point $z'\in \gamma$, let us consider the Green function $G_{\Omega}(z',\cdot)$. By the property of Green function that $\mathop{\sup}\limits_{\Omega}G_{\Omega}(z',\cdot)=0$, there exists a sequence of points $\{z_k\}\subset \Omega\backslash \gamma$ such that $\tilde g(z_k)\not=0$ for any $k$ and
$$\lim_{k\rightarrow +\infty}G_{\Omega}(z',z_k)\rightarrow 0.$$ 
Note that $\gamma$ is a compact subset of $\Omega\backslash\{z_k\}$, therefore it follows from Lemma \ref{lemma3} and  $G_{\Omega}(\cdot,z_k)$ are negative harmonic functions on $\Omega\backslash\{z_k\}$ that  $G_{\Omega}(z,z_k)\rightarrow 0$ uniformly on a neighborhood  of $\gamma$.
Note that $\widetilde d=\frac{\partial-\bar\partial}{\sqrt{-1}}$.  Using Lemma \ref{lemma4} and the dominated convergence theorem, we have 
 \begin{equation}
 	\label{1109c}\lim_{k\rightarrow+\infty}G_{\gamma}(z_k)=\int_{\gamma}\widetilde{d}G_{\Omega}(\cdot,z_k)=0.
 	 \end{equation}
Combining equality \eqref{eq:1118a} and \eqref{1109c}, we get that 
$$\frac{1}{2\pi}G_{\gamma}(z)\in\mathbb{Z}$$
 for any $z\in\Omega\backslash\gamma$.  For any $z_0\in\Omega$, note that any $\alpha\in\pi_1(\Omega,\tilde z_0)$ $(\tilde z_0\not=z_0)$ can be represented by a piecewise smooth curve $\gamma$ in $\Omega$ satsifying $z_0\not\in \gamma$. Thus, it follows from Proposition \ref{thm1} that 
$$\chi_{z_0}=1,$$
i.e.,  there exists a holomorphic function $\tilde f_{z_0}$ on $\Omega$ such that $|\tilde f_{z_0}|=e^{G_{\Omega}(\cdot,z_0)}$, which implies that $\Omega$ is conformally equivalent to the unit disc less a (possible) closed set of inner capacity zero (see \cite{Suita}). Using equality \eqref{eq:1109a}, we know that $\chi_u=1$, hence there exists a holomorphic function $g$ on $\Omega$ such that $v=\log|g|$.

Thus, Theorem \ref{thm4} holds.

\subsection{Proof of Theorem \ref{remark6}.}\label{proof2}\quad\par
For any $\alpha\in\Pi_1$, note that it can be represented by  a closed piecewise smooth curve $\gamma_{\alpha}$. Using Lemma \ref{lemma14}, there exists a harmonic function $\tilde u_{\alpha}$ on $\Delta$ such that
\begin{equation}\label{eq:1118b}
	\frac{1}{2\pi}\left(\tilde u_{\alpha}(z)-\int_{\gamma_{\alpha}}\widetilde{d}G_{\Omega}(\cdot,p(z))\right)\in\mathbb{Z}
\end{equation}
for any $z\in \Delta\backslash{p^{-1}(\gamma_{\alpha})}$.
Note that $d\widetilde{d}=2i\partial\bar\partial$ and $	\frac{1}{2\pi}\int_{\gamma_{0}}\widetilde{d}G_{\Omega}(\cdot,z_0)\in\mathbb{Z}$, where $\gamma_0\subset U_{z_0}\backslash\{z_0\}$ is a closed piecewise smooth curve and $U_{z_0}$ is a coordinate disc centered at $z_0$, hence we know that $\tilde u_{\alpha}$ is independent of the selection  of $\gamma_{\alpha}$. 
 Denote that 
 \begin{equation}
 	\label{eq:1118f}u_{\alpha}=\frac{1}{2\pi}\left((k+1)\tilde u_{\alpha}+\int_{\gamma_{\alpha}}\widetilde{d}u\right)
 \end{equation}
for any $\alpha\in \Pi_1,$ thus we have $\beta^*u_{\alpha}-u_{\alpha}$ is a integer constant function for any $\beta\in\pi_1(\Omega)$ by formula \eqref{eq:1118b}.  

Suppose that equality \eqref{eq:1118c} holds for $z_0\in \Omega$ with $v(z_0)>-\infty$. By Theorem \ref{thm3}, we know that
\begin{equation}\label{eq:1118d}
	\chi_{-u}=(\chi_{z_0})^{k+1}.
\end{equation} 
 Without loss of generality, assume that $z_0\not\in\gamma_{\alpha}$.  It follows from Proposition \ref{thm1}, equality \eqref{eq:1118d} and \eqref{eq:1118b} that 
\begin{equation}
	\nonumber
	\begin{split}
		u_{\alpha}(\tilde z_0)=&\frac{1}{2\pi}\left((k+1)\tilde u_{\alpha}(\tilde z_0)+\int_{\gamma_{\alpha}}\widetilde{d}u\right)\\
		=&\frac{1}{2\pi}\left((k+1)\tilde u_{\alpha}(\tilde z_0)-(k+1)\int_{\gamma_{\alpha}}\widetilde{d}G_{\Omega}(\cdot,z_0)\right)\\
		&+\left((k+1)\int_{\gamma_{\alpha}}\widetilde{d}G_{\Omega}(\cdot,z_0)+\int_{\gamma_{\alpha}}\widetilde{d}u\right)\\
		\in&\mathbb{Z},
	\end{split}
\end{equation}
where $\tilde z_0\in p^{-1}(z_0)$.

On the other hand, suppose that $u_{\alpha}(\tilde z_0)\in\mathbb{Z}$ for $\tilde z_0\in  p^{-1}(z_0),$ where $v(z_0)>-\infty$.  Without loss of generality, assume that $z_0\not\in\gamma_{\alpha}$. As 	$\frac{1}{2\pi}\left(\tilde u_{\alpha}(\tilde z_0)-\int_{\gamma_{\alpha}}\widetilde{d}G_{\Omega}(\cdot,z_0)\right)\in\mathbb{Z}$, it follows from Proposition \ref{thm1} and $u_{\alpha}(\tilde z_0)=\frac{1}{2\pi}\left((k+1)\tilde u_{\alpha}(\tilde z_0)+\int_{\gamma_{\alpha}}\widetilde{d}u\right)\in\mathbb{Z}$ that 
\begin{equation}
	\label{eq:1118e}\chi_{-u}(\alpha)=(\chi_{z_0}(\alpha))^{k+1}.
\end{equation}
As $\pi_1(\Omega)$ is generated by $\Pi_1$, equality \eqref{eq:1118e} shows that $(\chi_{z_0})^{k+1}=\chi_{-u}$, which implies that equality \eqref{eq:1118c} holds by using Theorem \ref{thm3}.

In the following part, we prove the characterization for $\{u_{\alpha}\}_{\alpha\in\Pi_1}$ being constant functions.

Firstly, suppose that $\Omega$ is conformally equivalent to the unit disc less a (possible) 
closed set of inner capacity zero and $v=\log|g|$, where $g\not\equiv0$ is a holomorphic function on $\Omega$. By formula \ref{eq:1118b}, we know $\tilde u_{\alpha}\in 2\pi\mathbb{Z}$. As $v=\log|g|=\log|\tilde g|+u$, we have $\chi_{u}=1$, which implies that $\int_{\gamma_{\alpha}}\widetilde{d}u\in2\pi\mathbb{Z}$ by Proposition \ref{thm1}. Thus, we know that $\{u_{\alpha}\}_{\alpha\in\Pi_1}$ are integer constant functions.

Now, assume that $\{u_{\alpha}\}_{\alpha\in\Pi_1}$ are integer constant functions. Then we have that equality \eqref{eq:1118c} holds for any $z\in\Omega$ with $v(z_0)>-\infty$. Then it follows from Theorem \ref{thm3} that $\Omega$ is conformally equivalent to the unit disc less a (possible) 
closed set of inner capacity zero and $v=\log|g|$, where $g\not\equiv0$ is a holomorphic function on $\Omega$.

Thus, Theorem \ref{remark6} has been proved.

\section{Proofs of Theorem \ref{thm5}, Corollary \ref{coro1} and Remark \ref{remark7}}

In this section, we prove Theorem \ref{thm5}, Corollary \ref{coro1} and Remark \ref{remark7}.

\subsection{Proof of Theorem \ref{thm5}.}\quad\par

Let $u$ be a harmonic function on $\Omega$.
By Lemma \ref{lemma2}, a point $z_0\in\Omega$ satisfies 
$$(\chi_{z_0})^{k+1}=\chi_{-u}$$ 
if and only if 
$$\left(e^{i\int_{\gamma_j}\frac{\partial G_{\Omega}(z,z_0)}{\partial n}ds(z)}\right)^{k+1}=e^{-i\int_{\gamma_j}\frac{\partial u(z)}{\partial n}ds(z)}$$ 
for $1\le j\le n-1$, that is, there exists some integers $l_j$ such that
\begin{equation}
	\label{eq:1118g}\int_{\gamma_j}\frac{\partial G_{\Omega}(z,z_0)}{\partial n}ds(z)=\frac{1}{k+1}\left(2\pi l_j-\int_{\gamma_j}\frac{\partial u(z)}{\partial n}ds(z)\right)
\end{equation} 
for $1\le j\le n-1$.

Since the boundary of $\Omega$ is analytic, $G_{\Omega}(z,z_0)$ can be extended to a harmonic function on $U\backslash\{z_0\}$, where $U$ is a neighborhood of $\bar{\Omega}$, and we can choose $\gamma_j$ to approximate $\Gamma_j$ arbitrarily. Hence 
$$\int_{\Gamma_j}\frac{\partial G_{\Omega}(z,z_0)}{\partial n}ds(z)=\int_{\gamma_j}\frac{\partial G_{\Omega}(z,z_0)}{\partial n}ds(z)$$ 
for arbitrary such $\gamma_j$. 
It follows from Remark \ref{remark4} and 
 Green third formula (see Lemma \ref{lemma1}) that  $$u_j(w)=\frac{1}{2\pi}\int_{\Gamma_j}\frac{\partial G(z,w)}{\partial n}ds(z)$$
 for any $1\le j\le n-1$ and any $w\in\Omega$, where $u_j$ is the solution of Dirichlet problem in Theorem \ref{thm5}. Using equality \eqref{eq:1118g}, we know that  a point $z_0\in\Omega$ satisfies 
 $$(\chi_{z_0})^{k+1}=\chi_{-u}$$ 
 if and only if 
 \begin{equation}
 	\label{eq:1118h}(k+1)u_j(z_0)+\frac{1}{2\pi}\int_{\gamma_j}\frac{\partial u(z)}{\partial n}ds(z)\in\mathbb{Z}
 \end{equation}
for any $1\le j\le n-1$.
Using Theorem \ref{thm3}, we get that Theorem \ref{thm5} holds.

\subsection{Proof of Corollary \ref{coro1}.}\quad\par
$u_j$ is the solution of Dirichlet problem in Theorem \ref{thm5} for any $1\le j\le n-1$. 
By Theorem \ref{thm5} (here we set $v=0$), if $z_0\in\Omega$ satisfies
$$(c_{\beta}(z_0))^{2(k+1)}=\left(\int_0^{+\infty}c(t)e^{-t}dt\right)\frac{\pi }{a}B_{\Omega,\rho_{z_0}}^{(k)}(z_0),$$ 
then there exists an integer $m_j$ such that
$$u_j(z_0)=\frac{m_j}{k+1}$$ 
for $1\le j\le n-1$. By the maximum principle, we have $m_j>0$ for any $1\le j\le n-1$. Denote that $u_n=1-\left(\sum_{j=1}^{n-1}u_j\right)\in C(\overline{\Omega})$, which is harmonic on $\Omega$.
By the maximum principle, it is clear that that $$u_n(z)\in[0,1]$$
 for any $z\in\overline{\Omega}$.
Note that $\sum_{j=1}^n u_j(z)=1$ on $\overline{\Omega}$. There exists a positive integer $m_n$ such that $u_n(z_0)=\frac{m_n}{k+1}$. Therefore we have
$$1=u_1(z_0)+u_2(z_0)+...+u_n(z_0)\ge\frac{n}{k+1},$$
which implies that $k\ge n-1$. Thus, Corollary \ref{coro1} holds.

\subsection{Proof of Remark \ref{remark7}.}\label{proof3}\quad\par
By solving Dirichlet Problem, there is $u_1\in C(\overline{\Omega})$ satisfying that $u_1$ is harmonic on $\Omega$, $u_1=1$ on $\Gamma_1$ and $u_1=0$ on $\Gamma_2\cup\Gamma_3$. Similarly,  
there is $u_2\in C(\overline{\Omega})$ satisfying that $u_2$ is harmonic on $\Omega$, $u_2=1$ on $\Gamma_2$ and $u_2=0$ on $\Gamma_1\cup\Gamma_3$.
It follows from Lemma \ref{l:u1u2} that there exists $z_0\in \Omega$ such that  $u_1(z_0)$ and $u_2(z_0)$ are rational numbers, then there exist a nonnegative integer $k$ such that $ku_1(z_0)$ and $ku_2(z_0)$ are both integers. By Theorem \ref{thm5}, we have
$$(c_{\beta}(z_0))^{2(k+1)}=\left(\int_0^{+\infty}c(t)e^{-t}dt\right)\frac{\pi }{a}B_{\Omega,\rho_{z_0}}^{(k)}(z_0).$$

\section{Appendix}
Let $\Omega$ be an open Riemann surface, which admits a nontrivial Green function $G_{\Omega}$. It is known that $G_{\Omega}(z,z')$ is a harmonic function with respect to $z\in\Omega\backslash\{z'\}$ when $z'\in\Omega$ is fixed and is a harmonic function with respect to $z'\in\Omega\backslash\{z\}$ when $z\in\Omega$ is fixed. From this, we can conclude that:

\begin{Lemma}
	\label{lemma12}
	$G_{\Omega}(z,z')$ is a smooth function on $(\Omega\times\Omega)\backslash Diag_{\Omega}$, where $Diag_{\Omega}=\{(z,z)|z\in\Omega\}$. 
\end{Lemma}

\begin{proof}
	For an arbitrary point $(z_0,z_0')\in(\Omega\times\Omega)\backslash Diag_{\Omega}$, choose local coordinate neighborhoods $V_{z_0}$ of $z_0$ and $V_{z_0'}$ of $z_0'$ such that $\overline{V}_{z_0}\cap\overline{V}_{z_0'}=\emptyset$. After identifying $V_{z_0}$ and $V_{z_0'}$ with their images under local coordinate maps, we may assume that $V_{z_0}$ and $V_{z_0'}$ are disjoint open subsets of $\mathbb{C}$.
	
	Let $\xi(z)=\exp\left(\frac{1}{|z|^2-1}\right)$ for $|z|<1$ and $\xi(z)=0$ for $|z|\ge1$, then $\xi(z)$ is a smooth function depending only on $|z|$ and $supp\xi=\{z\in C:|z|\le 1\}$. Let $\mu(z):=\xi(z)/\int_{\mathbb{C}}\xi(w)dV(w)$, where $dV$ denotes the Lebesgue measure. Denote that $$\mu_{\epsilon}(z):=\frac{1}{\epsilon^2}\mu\left(\frac{z}{\epsilon}\right),$$ where $0<\epsilon<1$, then $\mu_{\epsilon}(z)$ is a smooth function depending only on $|z|$, $supp \mu_{\epsilon}=\{z\in\mathbb{C}:|z|\le\epsilon\}$ and $\int_{\mathbb{C}}\mu_{\epsilon}(z)dV(z)=1$.  
	
	Let $z_1\in V_{z_0}$ be a fixed point, then $G_{\Omega}(z_1,z')$ is a continuous function with respect to $z'\in\overline{V}_{z_0'}$, hence $G_{\Omega}(z,z')$ is bounded on $\{z_1\}\times\overline{V}_{z_0'}$.
	For a fixed point $z_1'\in \overline{V}_{z_0'}$, $G_{\Omega}(z,z_1')$ is a negative harmonic function on a neighborhood of $\overline{V}_{z_0}$, $G_{\Omega}(z,z')$ is therefore bounded on $\overline{V}_{z_0}\times\overline{V}_{z_0'}$ by Lemma \ref{lemma3}. 
	Take $\epsilon<\frac{1}{2}\min\{dist(z_0,\partial{V_{z_0}}),dist(z_0',\partial{V_{z_0'}})\}$, then we have
	\begin{align}  
		G_{\Omega}(z,z')&=\int_{\mathbb{C}}G_{\Omega}(z-w,z')\mu_{\epsilon}(w)dV(w)\notag\\
		&=\int_{\mathbb{C}}\int_{\mathbb{C}}G_{\Omega}(z-w,z'-w')\mu_{\epsilon}(w)\mu_{\epsilon}(w')dV(w)dV(w')\notag\\
		&=\int_{\mathbb{C}}\int_{\mathbb{C}}G_{\Omega}(w,w')\mu_{\epsilon}(z-w)\mu_{\epsilon}(z'-w')dV(w)dV(w')\notag
	\end{align}
	for $(z,z')\in B(z_0,\epsilon)\times B(z_0',\epsilon)$, which implies that $G_{\Omega}(z,z')$ is smooth on a neighborhood of $(z_0,z_0')$. Since $(z_0,z_0')\in(\Omega\times\Omega)\backslash Diag_{\Omega}$ is arbitrarily chosen, $G_{\Omega}(z,z')$ is smooth on $(\Omega\times\Omega)\backslash Diag_{\Omega}$. 
\end{proof}

Let $w$ be a local coordinate on $U\subset\Omega$. The following lemma shows the smoothness of $G_{\Omega}(z,z')-\log|w(z)-w(z')|$ on $U\times U$.

\begin{Lemma}\label{l:1119}
	 $G_{\Omega}(z,z')-\log|w(z)-w(z')|$ can be extended to be a smooth function on $U\times U$.
\end{Lemma}
\begin{proof}
	Without loss of generality, we can assume that $w(U)\subset\mathbb{C}$ is unit disc $\Delta$.
	
	For each $z_0'\in U$, $G_{\Omega}(z,z_0')-\log|w(z)-w(z_0')|$ is a harmonic function on $U\backslash\{z_0'\}$ and is bounded near $z_0'$, hence can be extended to be a harmonic function on $U$. Since $z_0'\in\Delta$ is arbitrarily chosen, $G_{\Omega}(z,z')-\log|w(z)-w(z')|$ can be
	extended to be a function on $U\times U$ such that $G_{\Omega}(z,z_0')-\log|w(z)-w(z_0')|$ is a harmonic function on $U$ for any fixed $z_0'\in U$. Denote that 
	$$H(z,z'):=G_{\Omega}(z,z')-\log|w(z)-w(z')|.$$
	
Notice it that $G_{\Omega}(z,z')$ and $\log|w(z)-w(z')|$ are symmetric with respect to its two variables, then we know that
	$G_{\Omega}(z_0,z)-\log|w(z_0)-w(z)|$ is also a harmonic function on $U$ for any fixed $z_0\in U$.
	
	Recall that for unit disc, the Green function is 
	$$G_{\Delta}(z,z')=\frac{1}{2}\log\frac{|z-z'|^2}{|z-z'|^2+(1-|z|^2)(1-|z'|^2)}$$ 
	(see \cite{demailly1}). Note that $G_{\Omega}(z,z')<0$, then by the property of Green function, we have
	$$G_{\Omega}(z,z')\le G_{\Delta}(w(z),w(z'))$$ 
	on $U \times U$, which implies that 
	$$G_{\Omega}(z,z')-\log|w(z)-w(z')|+\frac{1}{2}\log\left(|w(z)-w(z')|^2+(1-|w(z)|^2)(1-|w(z')|^2)\right)\le0$$
	on $U \times U$. As $\frac{1}{2}\log\left(|w(z)-w(z')|^2+(1-|w(z)|^2)(1-|w(z')|^2)\right)$ is smooth on $U\times U$, we have $H(z,z')$ is bounded from above on $\overline{V}_{z_0}\times \overline{V}_{z_0}$, where $z_0$ is any point in $U$ and $V_{z_0}\Subset U$ is a neighborhood of $z_0$.  
	For each fixed $z_0'\in U$, $H(z,z_0')$ is a harmonic function on a neighborhood of $\overline{V}_{z_0}$, $H(z,z')$ is therefore bounded on $\overline{V}_{z_0}\times\overline{V}_{z_0}$ by Lemma \ref{lemma3}. 
	By the same convolution method in Lemma \ref{lemma12}, we get that
	 $H(z,z')$ is a smooth function on a neighborhood of $(z_0,z_0)$. Since  smoothness is a local property, Lemma \ref{l:1119} has been proved.
\end{proof}

 Let $\gamma$ be a piecewise smooth closed curve in $\Omega$. Using Lemma \ref{lemma12}, we can get the following lemma. 

\begin{Lemma}
	\label{lemma9}
	$\int_{\gamma}\widetilde{d}G_{\Omega}(\cdot,z')$ is a harmonic function with respect to $z'$ on $\Omega\backslash  \gamma$, where $\widetilde{d}=\frac{\partial-\bar\partial}{i}$.
\end{Lemma}

In the following, we discuss the harmonic function $\int_{\gamma}\widetilde{d}G_{\Omega}(\cdot,z')$.
The following lemma shows that we only need to consider the curves which intersect with themselves at finite many points.

\begin{Lemma}
	\label{lemma13}
	Suppose that  $\gamma$ is a piecewise smooth closed curve in $\Omega$. Then there exists a piecewise smooth closed curve $\tilde\gamma$ which is homotopic to $\gamma$ in $\Omega$, and $\tilde\gamma$ intersects with itself at finite many points.
\end{Lemma}

\begin{proof}
	$\gamma:[0,1]\rightarrow \Omega\backslash\{z_1\}$ is a piecewise smooth map, hence $[0,1]$ can be divided into finite many sections $I_i=[x_i-1,x_i]$, $1\le i\le N$, such that each $\gamma|_{I_i}$ is contained in a simply connected coordinate neighborhood $V_i$ in $\Omega$.
	
	Now construct a smooth curve $\tilde\gamma:[0,1]\rightarrow \Omega\backslash\{z_1\}$ as follows: $\tilde\gamma|_{I_1}:I_1\rightarrow V_1$ is a smooth mapping such that $\tilde\gamma(x_0)=\gamma(x_0)$, $\tilde\gamma(x_1)=\gamma(x_1)$ and $\tilde\gamma|_{I_1}$ does not intersect with itself unless $\gamma(x_0)=\gamma(x_1)$; 
	Suppose that $\tilde\gamma|_{I_i}$ has been defined for $1\le i\le k-1$, $k\le N$, then define $\tilde\gamma|_{I_k}:I_k\rightarrow V_k$ to be a smooth mapping such that $\tilde\gamma(x_{k-1})=\gamma(x_{k-1})$, $\tilde\gamma(x_k)=\gamma(x_k)$ and 
	$\tilde\gamma|_{I_k}$ intersects with $\tilde\gamma|_{I_i}$ at finite many points for $1\le i\le k$.
	
	 Since $V_i$ is simply connected, $\tilde\gamma|_{I_i}$ is homotopic to $\gamma|_{I_i}$ for any $1\le i\le N$. Thus, 
	$\tilde\gamma$ is homotopic to $\gamma$ and intersects with itself at finite many points.
\end{proof}

\begin{Lemma}
	\label{lemma14}
	Suppose that $\gamma$ is a piecewise smooth closed curve. Then there exists a harmonic function $H(z')$ on $\Delta$ such that $H(z')=\int_{\gamma}\widetilde{d} G_{\Omega}(\cdot,p(z'))+2k\pi$ for $z'\in\Delta\backslash p^{-1}(\gamma)$, where $k$ is an integer depending on $z'$ and $p:\Delta\rightarrow\Omega$ is the universal covering from unit disc $\Delta$ to $\Omega$.
\end{Lemma}

\begin{proof}
	Note that $d\widetilde{d}=2i\partial\bar\partial$ and $\int_{\gamma_{0}}\widetilde{d}G_{\Omega}(\cdot,z_0)\in2\pi\mathbb{Z}$, where $\gamma_0\subset U_{z_0}\backslash\{z_0\}$ is a closed piecewise smooth curve and $U_{z_0}$ is a coordinate disc centered on $z_0\in\Omega$. By Lemma \ref{lemma13}, we only need to consider  piecewise smooth closed curve $\gamma$ which intersects with itself at finite many points.
	
	Assume that $\gamma$ intersects with itself at $n$ points, and we prove this Lemma by induction on $n$. 
	
	When $n=1$, $\gamma$ is a simple closed curve. Firstly, suppose that $\Omega\backslash \gamma$ is not connected. Note that $\Omega$ is orientable, then we denote the two connected components of $\Omega\backslash\gamma$ by $\Omega_1$ and $\Omega_2$, where $\Omega_1$ is left to $\gamma$ and
	$\Omega_2$ is right to $\gamma$. Define
	$$h(z')=\int_{\gamma}\widetilde{d} G_{\Omega}(\cdot,z')$$
	for $z'\in\Omega_1$ and
	$$h(z')=\int_{\gamma}\widetilde{d} G_{\Omega}(\cdot,z')+2\pi$$
	for $z'\in\Omega_2$. $h(z')$ is a harmonic function on $\Omega\backslash\gamma$ by Lemma \ref{lemma9}. For any point $x_0\in \gamma$, choose a local coordinate neighborhood $V_{x_0}$ of $x_0$ and a segment $\gamma_1$ of $\gamma$ such that $\gamma_1\subset V_{x_0}$. Assume that $\gamma=\beta_1\gamma_1\beta_2$. 
	Construct a smooth curve $\gamma_2$ in $V_{x_0}\cap \Omega_2$ such that the starting point and end point of $\gamma_2$ coincide with $\gamma_1$. Denote the open set in $V_{x_0}\cap\Omega_2$ bounded by $\gamma_2\left(\gamma_1\right)^{-1}$ by $G$, where $\left(\gamma_1\right)^{-1}(t)=\gamma_1(1-t)$ on $[0,1]$. Note that $\tilde\gamma:=\beta_1\gamma_2\beta_2$ is a piecewise smooth closed curve in $\Omega$, then
	$\int_{\tilde{\gamma}}\widetilde{d} G_{\Omega}(\cdot,z')$ is a harmonic function on $\Omega_1\cup\gamma_1\cup G$, which is a connected component of $\Omega\backslash\tilde{\gamma}$.
	For $z'\in\Omega_1$, we have
	\begin{align} 
		\int_{\tilde{\gamma}}\widetilde{d} G_{\Omega}(\cdot,z')&=\int_{\gamma}\widetilde{d} G_{\Omega}(\cdot,z')+\int_{\gamma_2\left(\gamma_1\right)^{-1}}\widetilde{d} G_{\Omega}(\cdot,z')\notag\\
		&=\int_{\gamma}\widetilde{d} G_{\Omega}(\cdot,z')\notag\\
		&=h(z')\notag.
	\end{align}
	For $z'\in G\subset\Omega_2$, we have
	\begin{align} 
		\int_{\tilde{\gamma}}\widetilde{d} G_{\Omega}(\cdot,z')&=\int_{\gamma}\widetilde{d} G_{\Omega}(\cdot,z')+\int_{\gamma_2\left(\gamma_1\right)^{-1}}\widetilde{d} G_{\Omega}(\cdot,z')\notag\\
		&=\int_{\gamma}\widetilde{d} G_{\Omega}(\cdot,z')+2\pi\notag\\
		&=h(z')\notag.
	\end{align}
	As a result, $h(z')$ can be extended to a harmonic function on a neighborhood of $x_0$. Since $x_0\in \gamma$ is arbitrarily chosen, $h(z')$ can be extended to be a harmonic function on $\Omega$. 
	$H(z'):=h(p(z'))$ is therefore a harmonic function on $\Delta$ such that $$H(z')=\int_{\gamma}\widetilde{d} G_{\Omega}(\cdot,z')+2k\pi,$$
	 where $k=0$ if $z'\in p^{-1}(\Omega_1)$ and $k=1$ if $z'\in p^{-1}(\Omega_2)$. 
	
	Now suppose that $\Omega\backslash\gamma$ is connected. $d\left(\int_{\gamma}\widetilde{d} G_{\Omega}(\cdot,z')\right)$ is a  smooth closed differential form on $\Omega\backslash \gamma$. For any point $x_0\in \gamma$, choose a connected coordinate neighborhood $\Omega'$ of $x_0$ which can be divided by $\gamma\cap\Omega'$ into two connected components denoted by 
	$\Omega'_1$ and $\Omega'_2$, where $\Omega'_1$ is left to $\gamma\cap\Omega'$, $\Omega'_2$ is right to $\gamma\cap\Omega'$. Define
	$$h(z')=\int_{\gamma}\widetilde{d} G_{\Omega}(\cdot,z')$$
	for $z'\in\Omega'_1$ and
	$$h(z')=\int_{\gamma}\widetilde{d} G_{\Omega}(\cdot,z')+2\pi$$
	for $z'\in\Omega'_2$. 
	Similarly, $h(z')$ can be extended to a harmonic function on a neighborhood $\Omega'$ of $x_0$. Notice it that $d\left(\int_{\gamma}\widetilde{d} G_{\Omega}(\cdot,z')\right)=dh(z')$ on $\Omega'$, and $x_0\in \gamma$ is arbitrarily chosen.  $d\left(\int_{\gamma}\widetilde{d} G_{\Omega}(\cdot,z')\right)$ can be extended to be a smooth closed form on $\Omega$, and 
	we denote it by $\omega$. Then $p^*\omega$ is a smooth closed form on $\Delta$.
	
	Fix a connected component $U$ of $\Delta\backslash p^{-1}(\gamma)$, then $\int_{\gamma}\widetilde{d} G_{\Omega}(\cdot,p(z'))$ is a harmonic function on $U$, and its differential is $p^*\omega$.
	Since $\Delta$ is simply connected, $\int_{\gamma}\widetilde{d} G_{\Omega}(\cdot,p(z'))$ can be extended to a harmonic function denoted by $H(z')$ on $\Delta$ through $p^*\omega$. 
	
It is clear that $p(U)=\Omega\backslash\gamma$. Thus, for any point $z_1'\in\Delta\backslash p^{-1}(\gamma)$, there exists a point $z_2'\in U\cap p^{-1}(p(z_1'))$, such that $p(z_1')=p(z_2')$. Let $\tilde\gamma$ be a curve in $\Delta$ from $z_2'$ to $z_1'$. $p_*\tilde\gamma$ is a smooth closed curve in $\Omega$, and after a homotopy if necessary (similarly to the proof of Lemma \ref{lemma13}), we can assume that $p_*\tilde\gamma$ intersects $\gamma$ at finite many points.
	Divide $[0,1]$ into $I_j=[t_{j-1},t_j]$ for $1\le j\le N$, such that $p_*\tilde\gamma|_{(t_{j-1},t_j)}$ intersects $\gamma$ at one point and $p_*\tilde\gamma(t_j)\notin \gamma$. By the construction of $\omega$ and $H(z')$, we have
	\begin{align}
		H(\tilde\gamma(t_j))&=H(\tilde\gamma(t_{j-1}))+\int_{\tilde\gamma|_{I_j}}p^*\omega \notag \\
		&=H(\tilde\gamma(t_{j-1}))+\int_{p_*\tilde\gamma|_{I_j}}\omega \notag \\
		&=H(\tilde\gamma(t_{j-1}))+\int_{\gamma}\widetilde{d} G_{\Omega}(\cdot,p_*\tilde\gamma(t_j))-\int_{\gamma}\widetilde{d} G_{\Omega}(\cdot,p_*\tilde\gamma(t_{j-1}))+2k_j\pi\notag,
	\end{align}
	where $k_j=1$ or $-1$, depending on whether $p_*\tilde\gamma|_{I_{j}}$ crossing $p_*\tilde\gamma$ from right to left or not. Since $p_*\tilde\gamma(t_N)=p_*\tilde\gamma(t_0)$,
	\begin{align}
		H(z_1')-H(z_2')&=\sum_{i=1}^N (H(\tilde\gamma(t_i))-H(\tilde\gamma(t_{i-1})))\notag\\
		&=\sum_{i=1}^N \left(\int_{\gamma}\widetilde{d} G_{\Omega}(\cdot,p_*\tilde\gamma(t_j))-\int_{\gamma}\widetilde{d} G_{\Omega}(\cdot,p_*\tilde\gamma(t_{j-1}))+2k_j\pi\notag\right)\\
		&=\sum_{i=1}^N 2k_i\pi\notag.
	\end{align}
	Since $z_2'\in U$ and $p(z_2')=p(z_1')$, we have  $H(z_2')=\int_{\gamma}\widetilde{d} G_{\Omega}(\cdot,p(z_2'))=\int_{\gamma}\widetilde{d} G_{\Omega}(\cdot,p(z_1'))$. Therefore 
	$$H(z_1')=\int_{\gamma}\widetilde{d} G_{\Omega}(\cdot,p(z_1'))+2k\pi,$$ where $k=\sum_{j=1}^N k_j$ ia an integer.
	
	Suppose that this Lemma has been proved when $n=k-1$, $k\ge 2$. Assume that $\gamma$ intersects with itself at $k$ points. Denote $\gamma(t_1)$ to be the first point of $\gamma$ at which $\gamma$ intersects with itself, and $\gamma(t_1)=\gamma(t_0)$, $0\le t_0<t_1$. 
	
	Let $\tilde\gamma_1(t):[0,t_1-t_0]\rightarrow \Omega$ satisfy $\tilde\gamma_1(t)=\gamma(t-t_0)$; 
	Let $\tilde\gamma_2(t):[0,1-t_1+t_0]$ satisfy $\tilde\gamma_2(t)=\gamma(t)$ on $[0,t_0]$ and $\tilde \gamma_2(t)=\tilde\gamma(t+t_1-t_0)$ on $[t_0,1-t_1+t_0]$.
	Note that $\tilde\gamma_1$ is a smooth simple closed curve and $\tilde\gamma_2$ is a smooth closed curve which intersects with itself at most $k-1$ points. By assumption, for $1\le j\le2$, there exists a harmonic function $H_j$ on $\Delta$
	such that for any point $z'\in\Delta\backslash p^{-1}(\tilde\gamma_j)$, $$H_j(z')=\int_{\tilde\gamma_j}\widetilde{d} G_{\Omega}(\cdot,p(z'))+2k_j'\pi,$$ where $k_j'$ is an integer depending on $z'$.
		Define $H(z')=H_1(z')+H_2(z')$ on $\Delta\backslash p^{-1}(\gamma)$,
	\begin{align}
		H_1(z')+H_2(z')&=\int_{\tilde\gamma_1}\widetilde{d} G_{\Omega}(\cdot,p(z'))+2k_1'\pi+\int_{\tilde\gamma_2}\widetilde{d} G_{\Omega}(\cdot,p(z'))+2k_2'\pi\notag\\
		&=\int_{\gamma}\widetilde{d} G_{\Omega}(\cdot,p(z'))+2(k_1'+k_2')\pi\notag,
	\end{align}
	where $k_1'$ and $k_2'$ are two integers depending on $z'$, which completes the induction.
	
	Thus, Lemma \ref{lemma14} has been proved.
\end{proof}

\begin{Remark}\label{remark:example}
	From the proof of Lemma \ref{lemma14}, we can see that the following two statements are equivalent:
	
	$(1)$ each simple closed curve divides $\Omega$ into two disconnected sets;  
	
	$(2)$ for every piecewise smooth closed curve $\gamma$, there exists a harmonic function $h(z')$ on $\Omega$ such that 
	$h(z')=\int_{\gamma}\widetilde{d} G_{\Omega}(\cdot,z')+2k\pi$ for $z'\in\Omega\backslash \gamma$, where $k$ is an integer depending on $z'$.
	
It is clear that statement $(1)$ holds for any open subsets of $\mathbb{C}$,  and  statement $(1)$ does not hold when $\Omega$ is a torus less a small closed disc.
\end{Remark}


\vspace{.1in} {\em Acknowledgements}. The authors would like to thank  Dr. Shijie Bao and Dr. Zhitong Mi for checking the manuscript and  pointing out some typos. The first named author was supported by National Key R\&D Program of China 2021YFA1003100, NSFC-11825101, NSFC-11522101 and NSFC-11431013.

\bibliographystyle{references}
\bibliography{xbib}

\begin{thebibliography}{100}
	
\bibitem{Ahlfors74}L.V. Ahlfors and L. Sario, Riemann Surfaces, Princeton University Press, Princeton, NJ, 1974.


\bibitem{Berg70}S. Bergman, The kernel function and the conformal mapping. Revised ed. Providence (RI): American Mathematical Society, 1970. (Mathematics Surveys; $\uppercase\expandafter{\romannumeral5}$)
\bibitem{Blocki}Z. Blocki, Suita conjecture and the Ohsawa-Takegoshi extension theorem, Invent. Math. 193(2013), 149-158.
Math. 193 (2013) 149--158.
\bibitem{demailly1}J.-P Demailly, Complex analytic and differential geometry, electronically accesssiable at https://www-fourier.ujf-grenoble.fr/\textasciitilde
 demailly/manuscripts/agbook.pdf.
\bibitem{evans}L. Evans, Partial differential equations, Graduate Studies in Mathematics, American Mathematical Society, Berkely, 1997.
\bibitem{fulton}W. Fulton, Algebraic topology a first course, Graduate texts in Mathematics, Springer-Verlag New York. Inc, 1995.

\bibitem{OF81}O. Forster, Lectures on Riemann surfaces, Grad. Texts in Math., 81, Springer-Verlag, New York-Berlin, 1981.

\bibitem{G-M}Q.A. Guan and Z.T. Mi, Concavity of minimal $L^2$ integrals related to multiplier ideal sheaves, Peking Math J (2022). https://doi.org/10.1007/s42543-021-00047-5.



\bibitem{G-Y}Q,A. Guan and Z. Yuan, Concavity property of minimal $L^2$ integrals with Lebesgue measurable gain, preprint. https://www.reserachgate.net/publication/353794984.


\bibitem{G-M-Y}Q.A. Guan, Z.T Mi and Z.Yuan, Concavity property of minimal $L^2$ integrals with lebesgue measurable gain $\uppercase\expandafter{\romannumeral2}$. http://www.researchgate.net/publication/354464147.
\bibitem{G-ZhouL2_CR}Q.A. Guan, X.Y. Zhou, Optimal constant problem in the $L^2$ extension theorem. C. R. Math. Acad. Sci. Paris 350 (2012), no. 15--16, 753--756.

\bibitem{guan-zhou13ap}Q.A. Guan and X.Y. Zhou, A solution of an $L^{2}$ extension problem with an optimal estimate and applications,
Ann. of Math. (2) 181 (2015), no. 3, 1139--1208.


\bibitem{lz}Z. Li, $L^2$ extension theorem and the approximation of Bergman spaces. PhD thesis, Institute of
Mathematics, Chinese Academy of Sciences (2019)
\bibitem{nadel}A. Nadel, Multiplier ideal sheaves and K$\ddot{a}$hler-Einstein metrics of positive scalar curvature, Ann. of Math. (2) 132 (3) (1990) 549-596.
\bibitem{nehari}Z. Nehari, Conformal mapping, Dover Publication, Inc. New York,1975.

\bibitem{L-K}L. Sario and K. Oikawa, Capacity functions, Grundl. Math. Wissen. 149, Springer-Verlag, New York, 1969. Mr 0065652. Zbl 0059.06901.
\bibitem{Suita}N. Suita, Capacities and kernels on Riemann surface, Arch. Rational Mech. Anal. 46 (1972), 212-217.
\bibitem{x-z}W. Xu, X.Y. Zhou, Optimal $L^2$ extension of openness type, arxiv:2202.04791.  
\bibitem{yamada}A. Yamada, Topics related to reproducing kernels, theta functions and the Suita conjecture (Japanese), The theory of reproducing kernels and their applications (Kyoto 1998).
\end{thebibliography}

\end{document}